\documentclass [11pt]{article}
 \usepackage{amssymb}
 \usepackage{amsmath,amsthm}
 \usepackage{mathrsfs}
 \usepackage{color}
\usepackage[numbers,sort&compress]{natbib}
 \usepackage{graphicx}
\usepackage{tikz}
 \usepackage[colorlinks=true]{hyperref}

\normalsize

\renewcommand{\dfrac}[2]{\lower0.15ex\hbox{\large$\textstyle\frac{#1}{#2}$}}

\newtheorem{theorem}{Theorem}[section]
\newtheorem{lemma}[theorem]{Lemma}

\newtheorem{corollary}[theorem]{Corollary}

\numberwithin{equation}{section}

\begin{document}
\title
{\bf A note on the random triadic process 
\thanks{ The work was partially supported by NSFC(No.12071274).}}

\author
{{Fang Tian$^1$\qquad Yiting Yang$^2$}\\
{\small $^1$Department of Applied Mathematics}\\%
{\small Shanghai University of Finance and Economics, Shanghai, 200433, China} \\
{\small\tt tianf@mail.shufe.edu.cn}\\[1ex]
{\small $^2$School of Mathematical Sciences}\\
{\small Tongji University, Shanghai, 200092, China}\\
{\small\tt ytyang@tongji.edu.cn}
}

\date{}
 \maketitle

\begin{abstract}

For a fixed integer $r\geqslant 3$, let $\mathbb{H}_r(n,p)$
be a random $r$-uniform hypergraph on the vertex set $[n]$,
where each $r$-set is an edge randomly and independently with probability~$p$.
The random $r$-generalized triadic process starts with
a complete bipartite graph $K_{r-2,n-r+2}$ on the same vertex set,
chooses two distinct vertices $x$ and $y$  uniformly at random
 and iteratively adds $\{x,y\}$ as an edge if there is a subset $Z$ 
 with size $r-2$, denoted as $Z=\{z_1,\cdots,z_{r-2}\}$, such that 
 $\{x,z_i\}$ and $\{y,z_i\}$ for $1\leqslant i\leqslant r-2$
 are already edges in the graph and $\{x,y, z_1,\cdots,z_{r-2}\}$ 
 is an edge in $\mathbb{H}_r(n,p)$. The random triadic process 
 is an abbreviation for the random $3$-generalized triadic process. 
 Kor\'{a}ndi et al. proved a sharp threshold probability for 
 the propagation of the random triadic process, that is, 
 if $p= cn^{ - \frac 12}$ for some positive constant $c$, 
 with high probability, the triadic process reaches the complete graph
when $c> \frac 12$ and stops at $O(n^{\frac 32})$ edges when $c< \frac 12$.
In this note, we consider the final size of the random $r$-generalized 
triadic process when $p=o( n^{- \frac 12}\log^{ \alpha(3-r)} n)$
with a constant $\alpha> \frac 12$. We show that 
the generated graph of the process essentially behaves like $\mathbb{G}(n,p)$.
The final number of added edges in the process, with high probability, equals
$ \frac {1}{2}n^{2}p(1\pm o(1))$ provided that $p=\omega(n^{-2})$. 
The results partially complement the ones on the case of $r=3$.
\end{abstract}



\section{Introduction}

For a fixed integer $r\geqslant 3$, the random $r$-generalized triadic 
process starts with a complete bipartite graph $K_{r-2,n-r+2}$ on the vertex set $[n]$
with one set of the vertex partition denoted as $\{v_1,\cdots,v_{r-2}\}$,
chooses two distinct vertices $x$ and $y$ uniformly at random
 and iteratively adds $\{x,y\}$ as an edge  if there is a subset $Z$ 
 with size $r-2$, denoted as $Z=\{z_1,\cdots,z_{r-2}\}$,
such that $\{x,z_i\}$ and $\{y,z_i\}$ for $1\leqslant i\leqslant r-2$
 are already edges in the graph and $\{x,y, z_1,\cdots,z_{r-2}\}$ 
 is an edge in $\mathbb{H}_r(n,p)$, where  $\mathbb{H}_r(n,p)$ 
 is a random $r$-uniform hypergraph on the same vertex set  
 and each $r$-set is an edge randomly and independently with probability $p$. 
 The random $r$-generalized triadic process can be regarded as 
 a simplistic model of the evolution of a social network
on the vertex set $[n]$, where a reliable friendship between two vertices can
only be formed if they share at least $r-2$ common friends and
the connection between these two vertices eventually occurs  with probability $p$.
In particular, for $r=3$, Kor\'{a}ndi et al.~\cite{korandi16}
firstly referred to the above process as the triadic process,
which is a variant of the triangle-free process~\cite{bohman09}
with a similar nature. 

Bohman~\cite{bohman09}  applied the differential equation
method to discuss the final size of the triangle-free process.
The differential equation method applies a pseudo-random heuristic 
to divine the trajectories of a group of graph parameters that govern 
the evolution of the process. Such heuristics also play a central role 
in our understanding of several other constrained random processes that produce 
interesting combinatorial objects~\cite{bennett15, bohman15,kuhn16}.

In terms of the tools developed by Bohman~\cite{bohman09},
for $r=3$ and $p= cn^{- \frac 12}$ with some positive constant $c$,
Kor\'{a}ndi et al.~\cite{korandi16} employed this method 
to prove a sharp threshold probability for the propagation 
of the random triadic process, that is,
with high probability (\textit{w.h.p.} for short), 
the triadic process reaches the complete graph
when $c> \frac 12$ and  stops at  $O(n^{ \frac 32})$ edges when $c< \frac 12$.
A generalization of the propagation threshold
 in~\cite{korandi16} was obtained by Morrison and
Noel~\cite{morr2021} in the hypergraph bootstrap process.

In this note, for any fixed integer $r\geqslant 3$, we consider 
the final size of the random $r$-generalized triadic process 
when $p=o( n^{- \frac 12}\log^{\alpha(3-r)} n)$ with a constant $\alpha> \frac 12$. 
The differential equation method  is not necessary 
under this condition. The proof essentially consists of
an iterative application of standard measure
concentration inequalities by adapting one part of the analysis 
when $c< \frac 12$  in~\cite{korandi16}. 
Let $u,v\in [n]\backslash\{v_1,\cdots,v_{r-2}\}$ be any two distinct vertices 
and $W\subseteq [n]\backslash\{u,v\}$ be a subset with size $r-2$.
At any step in the process, the $r$-set $\{u,v\}\cup W$ is 
called to be an \textit{open} $(r-2,2)$-walk at the vertex $u$ or the vertex $v$,
denoted as the notation $uWv$, if the subgraph it induces contains
 a complete bipartite graph $K_{2,r-2}$ with vertex partitions $\{u,v\}$
 and  $W$, where at least one of the edges in this $K_{2,r-2}$ is added 
 in the previous step, while the pair $\{u,v\}$ has not yet been considered
 to be  an edge. We call the corresponding $r$-set $\{u,v\}\cup W$ 
 of an open $(r-2,2)$-walk $uWv$ as an sample. An open $(r-2,3)$-walk from the vertex $u$ to 
 the vertex $v$, denoted as $uWw'v$, is obtained by attaching 
 the edge $w'v$ to an open $(r-2,2)$-walk $uWw'$ at the vertex $w'$;
 and an open $(r-2,4)$-walk between the vertex $u$ and the vertex $v$,
 denoted as $uW_1w'W_2v$, is obtained by sticking together two open $(r-2,2)$-walks
$uW_1w'$ and $w'W_2v$ at the vertex $w'$, where $W_1,W_2\subseteq [n]\backslash\{u,v\}$ are two
subsets with size $r-2$ and $w'\in [n]\backslash \{W_1\cup W_2\cup\{u,v\}\}$.

We will adopt the standard Landau notation and similar expressions from~\cite{korandi16}. 
For two positive-valued functions $f, g$ on the variable $n$,
we write  $f\ll g$ to denote $\lim_{n\rightarrow \infty}f(n)/g(n)=0$, $f\sim g$ to denote
$\lim_{n\rightarrow \infty}f(n)/g(n)=1$ and $f\lesssim g$ 
if and only if $\limsup_{n\rightarrow \infty} f(n)/g(n)\leqslant 1$.
For an event $A$ and a random variable $Z$ in an arbitrary 
probability space $(\Omega,\mathcal{F},\mathbb{P})$,
 $\mathbb{P}[A]$ and $\mathbb{E}[Z]$ denote the probability of  $A$
 and the expectation of  $Z$. If $X$ is the sum of $n$ independent variables, 
 each equal to $1$ with probability $p$ and $0$ otherwise,
 then we say that $X$ is a Binomially distributed random 
 variable with parameters $n$ and $p$, or just $X\sim\textbf{Bin}[n,p]$.
As usual, all asymptotics in this note are with respect to $n\to\infty$. 
We say some property holds \textit{w.h.p.} if the probability that it holds tends to $1$.
All logarithms are natural, and the floor and ceiling signs are omitted
whenever they are not crucial. Our main result is Theorem~1.1 in the following.
\begin{theorem}
Let $r\geqslant 3$ be a fixed integer and suppose that $p=o( n^{- \frac 12}\log^{\alpha(3-r)} n)$
with a constant $\alpha> \frac 12$. The final number of added edges 
in the random $r$-generalized triadic process, \textit{w.h.p.}, 
equals $ \frac {1}{2}n^{2}p(1\pm o(1))$  when $p=\omega(n^{-2})$.
\end{theorem}
\noindent Theorem~1.1  complements the results
in~\cite{korandi16} when $r=3$ to the case of $p=o( n^{- \frac 12})$.
\begin{corollary}
Suppose that $p=o( n^{- \frac 12})$. The final number of added edges 
in the random triadic process, \textit{w.h.p.}, equals
$ \frac {1}{2}n^{2}p(1\pm o(1))$  when $p=\omega(n^{-2})$.
\end{corollary}

Based on the proof of Theorem~1.1, it is natural to have
the threshold probability for the property that the graph generated 
by the process is still connected after removing $\{v_1,\cdots,v_{r-2}\}$
is the threshold probability for the connectivity of $\mathbb{G}(n,p)$. 
\begin{theorem}
Let $r\geqslant 3$ be a fixed integer. The threshold probability  of connectivity for
the generated graph by the random $r$-generalized triadic process 
after removing $\{v_1,\cdots,v_{r-2}\}$ 
 is $ n^{-1}\log n$.
\end{theorem}

\section{Main Results}

For any fixed integer $r\geqslant 3$ and $p=o( n^{- \frac 12}\log^{\alpha(3-r)} n)$
with a constant $\alpha> \frac{1}{2}$, the proof is simpler than the one in~\cite{korandi16}.
Note that $r$ is fixed, it is only necessary to consider the edges 
added between the vertices in $[n]\backslash\{v_1,\cdots,v_{r-2}\}$.
Since the order in which candidates for edges are considered does not matter,
to analyze the process in rounds is viable. That is, for any two 
distinct vertices $u,v\in [n]\backslash\{v_1,\cdots,v_{r-2}\}$
and $W\subseteq [n]\backslash \{u,v\}$ with size $r-2$, all  $r$-sets $\{u,v\}\cup W$ in these 
 open $(r-2,2)$-walks with the form $uWv$
 are  sampled at the same time, and the corresponding edges
  $\{u,v\}$ are added independently with probability $p$.
We will use standard measure concentration
inequalities to analyze the process according to the values of $p$.

Let $\textbf{G}(0)$ be the complete bipartite graph $K_{r-2,n-r+2}$ on the vertex set $[n]$
with one set of the vertex partition denoted as
 $\{v_1,\cdots,v_{r-2}\}$. For an integer $i\geqslant 1$, let us denote the graph
after the $i$-th round by $\textbf{G}(i)$.
For any two distinct vertices $u,v\in [n]\backslash\{v_1,\cdots,v_{r-2}\}$,
let $\textrm{D}_u(i)$ be the set of neighbors of the vertex $u$,
$\textrm{F}_u(i)$ be the set of open $(r-2,2)$-walks at the vertex $u$,
$\textrm{Y}_{{uv}}(i)$ be the set of open $(r-2,3)$-walks
from the vertex $u$ to the vertex $v$, $\textrm{X}_{uv}(i)$ and 
$\textrm{Z}_{uv}(i)$ be the set of  common neighbors
and open $(r-2,4)$-walks in $\textbf{G}(i)-\{v_1,\cdots,v_{r-2}\}$, respectively.

\begin{lemma} {\rm{(Chernoff bound~\cite{cher52})}}
Let  $X\sim\textbf{Bin}[n,p]$. Then, $\mathbb{P}[|X-np|>t]<2\exp\Bigl[- \frac{t^2}{3np}\Bigr]$
for any $0< t\leqslant np$ and $\mathbb{P}[X>np+t]<\exp\Bigl[- \frac{t^2}{2(np+t/3)}\Bigr]$ 
for any $t>0$.
\end{lemma}

\begin{proof}[Proof of Theorem~1.1]
For any two distinct vertices $u,v\in [n]/\{v_1,\cdots,v_{r-2}\}$, after round one,
the random graph $\textbf{G}(1)$ is the result of adding all edges $\{u,v\}$ with probability $p$,
 which implies that $\textbf{G}(1)-\{v_1,\cdots,v_{r-2}\}$ 
 is approximately equal to $\mathbb{G}(n,p)$
because $r$ is a fixed integer. 

We make no attempt to optimize all these constants in the following proof.

\vskip 0.2cm
\noindent{\bf{Case 1}}.\,
Consider the case when $p\geqslant n^{- \frac 78}$ and
$p=o( n^{- \frac 12}\log^{\alpha(3-r)} n)$
with a constant $\alpha> \frac 12$.

\vskip 0.2cm
\noindent{\bf Claim A}.\, For any given arbitrarily small constant $\epsilon>0$, 
with probability at least $1- \frac 2n$, we have $|\textrm{D}_u(1)|\sim np$,  
$|\textrm{F}_u(1)|<\epsilon n$, $|\textrm{Y}_{{uv}}(1)|< 
\epsilon np\log^{2\alpha(r-2)}n$, $|\textrm{X}_{uv}(1)|< 3\log n$
and $|\textrm{Z}_{uv}(1)|< \epsilon n\log^{2\alpha(r-2)}n$ hold for 
any two distinct vertices $u,v\in [n]\backslash\{v_1,\cdots,v_{r-2}\}$,
respectively. 

\begin{proof}[Proof of Claim A]\, For any vertex $u\in [n]/\{v_1,\cdots,v_{r-2}\}$,
we have $|\textrm{D}_u(1)|$ is a random variable  with 
$|\textrm{D}_u(1)|\sim\textbf{Bin}[n-r+1,p]$ and $\mathbb{E}[|\textrm{D}_u(1)|]\sim np$
because $r$ is a fixed integer. Given any arbitrarily small 
constant $\epsilon>0$, by Lemma 2.1 with $t= \epsilon np$, 
note that $p\geqslant n^{- \frac 78}$, then we have that 
\begin{equation*}
\mathbb{P}\Bigl[\bigl||\textrm{D}_u(1)|-np\bigr|>t\Bigr]
<2\exp\Bigl[-\frac{t^2}{3np}\Bigr]< \frac{1}{n^2}.
\end{equation*}
Taking a union bound of the above equation for all vertices,
 with probability at most $ \frac 1n$, $|\textrm{D}_u(1)|\sim np$ 
 does not hold for some vertex $u\in [n]\backslash\{v_1,\cdots,v_{r-2}\}$. 
Furthermore, for any two distinct vertices $u,v\in [n]/\{v_1,\cdots,v_{r-2}\}$, 
note that $|\textrm{X}_{uv}(1)|$ is  a random variable  with 
$|\textrm{X}_{uv}(1)|\sim\textbf{Bin}[n-r,p^2]$
and $\mathbb{E}[|\textrm{X}_{uv}(1)|]\sim np^2\ll \log n$. 
Choosing $t= 2 \log n$ in Lemma 2.1,  we also have 
\begin{equation*}
\mathbb{P}\Bigl[|\textrm{X}_{uv}(1)|>np^2+t\Bigr]<\exp\Bigl[- \frac{t^2}{2(np^2+t/3)}\Bigr]
< \frac{1}{n^{3}}.
\end{equation*}
By taking a union bound of the above equation for any pair of vertices,
with probability at most $ \frac 1n$, $|\textrm{X}_{uv}(1)|<3 \log n$ 
does not hold for some two distinct vertices $u,v\in [n]\backslash\{v_1,\cdots,v_{r-2}\}$. 

Assume that the events $|\textrm{D}_u(1)|\sim np$ and $|\textrm{X}_{uv}(1)|<3 \log n$
all hold for any two distinct vertices $u,v\in [n]\backslash\{v_1,\cdots,v_{r-2}\}$ below.
In order to obtain an open $(r-2,2)$-walk with the form  $uWv$
in $\textrm{F}_u(1)$, we first choose a neighbor of the vertex $u$, 
denoted as $w$, with at most $|\textrm{D}_u(1)|$ ways;
 and  choose a neighbor of the vertex $w$, denoted as $v$, 
 with at most $|\textrm{D}_w(1)|$ ways; then we take other $(r-3)$ common neighbors 
 of the vertex $u$ and the vertex $v$ with at most $|\textrm{X}_{uv}(1)|^{r-3}$ ways. 
Hence, for any  vertex $u\in [n]\backslash\{v_1,\cdots,v_{r-2}\}$, 
it follows that 
\begin{align*}
|\textrm{F}_u(1)|& \leqslant |\textrm{D}_u(1)|\cdot |\textrm{D}_w(1)|\cdot |\textrm{X}_{uv}(1)|^{r-3},
\end{align*}
and then $|\textrm{F}_u(1)|< (np)^2(3\log n)^{r-3}<\epsilon n$ 
because $p=o(n^{- \frac {1}{2}}\log^{\alpha(3-r)}n)$ with a constant $\alpha> \frac 12$.
Similarly, in order to generate an open $(r-2,3)$-walk with the form $uWw'v$,
we first choose a neighbor of the vertex $v$, denoted as $w'$, with at most  $|\textrm{D}_v(1)|$
ways; then we take $(r-2)$ common neighbors of the vertex $u$ and the vertex $w'$
with at most $|\textrm{X}_{uw'}(1)|^{r-2}$ ways. It  also follows that, 
for any two distinct vertices $u,v\in [n]\backslash\{v_1,\cdots,v_{r-2}\}$,
\begin{align*}
|\textrm{Y}_{{uv}}(1)|\leqslant |\textrm{D}_v(1)|\cdot |\textrm{X}_{uw'}(1)|^{r-2}.
\end{align*}
Thus, $|\textrm{Y}_{{uv}}(1)|< (np)(3\log n)^{r-2}<\epsilon np\log^{2\alpha(r-2)}n$ 
when $r\geqslant 3$ and $\alpha> \frac 12$. Finally, we also have
\begin{align*}
|\textrm{Z}_{uv}(1)| \leqslant |\textrm{F}_u(1)|\cdot |\textrm{X}_{vw}(1)|^{r-2}<\epsilon n(3\log n)^{r-2}<\epsilon n\log^{2\alpha(r-2)}n,
\end{align*}
where $w\in [n]\setminus \{u,v\}$ is a vertex of an open $(r-2,2)$-walk $uWw$ in $\textrm{F}_u(1)$.

Hence, above all, with probability at least $1- \frac 2n$, all these equations in Claim A
hold for any two distinct vertices $u,v\in [n]\backslash\{v_1,\cdots,v_{r-2}\}$.
\end{proof}

\vskip 0.2cm
\noindent{\bf Claim B}.\, Let $\ell= \log n$. We have, \textit{w.h.p.},  
$|\textrm{D}_u(i)|\sim np$ and $|\textrm{X}_{uv}(i)|\lesssim \log^{2\alpha} n$
hold for any two distinct vertices
$u,v\in [n]\backslash\{v_1,\cdots,v_{r-2}\}$
when $1\leqslant i\leqslant \ell$.

\begin{proof}[Proof of Claim B]\,
 In fact, we will show a stronger result, that is, given
 any arbitrarily small constant $\epsilon>0$, for any $1\leqslant  i\leqslant \ell$,
   with probability at least $(1- \frac 2n)^{i}$, the following inequalities
 \begin{align}
|\textrm{D}_u(j)|&\lesssim \Bigl[1+\bigl(1+ \exp[{-1}]+\cdots +\exp[{1-{j}}]\bigr)\epsilon
+ jn^{ \frac 56}p^{ \frac{11}{6}}+ j(np)^{- \frac 13}\Bigr]np,\\
|\textrm{F}_u(j)|&\lesssim \epsilon n \exp[{-j}]+(np)^{ \frac {11}{6}},\\
|\textrm{Y}_{{uv}}(j)|&\lesssim \epsilon np\log^{2\alpha(r-2)} n,\\
|\textrm{X}_{uv}(j)|&\lesssim \log^{2\alpha} n,\\
|\textrm{Z}_{uv}(j)|&<\epsilon n\log^{2\alpha(r-2)} n,
\end{align}
 hold in $\mathbf{G}(j)$ for any two distinct vertices 
$u,v\in [n]\backslash\{v_1,\cdots,v_{r-2}\}$ 
 when $1\leqslant j\leqslant i$. In particular, after $\ell$ rounds,
  that is, for $i=\ell$, it follows that, 
with probability at least $(1- \frac{2}{n})^{\log n}\rightarrow 1$,
the inequalities in $(2.1)$-$(2.5)$ all hold in $\mathbf{G}(j)$ for any two distinct vertices
$u,v\in [n]\backslash\{v_1,\cdots,v_{r-2}\}$
when $1\leqslant j\leqslant\ell$.

Now we  prove the above statement by induction on rounds. 
Obviously, these equations all hold when $i=1$ by Claim A. 
Assume that, for any  $1\leqslant i<\ell$, with probability at least $(1- \frac 2n)^{i}$, 
the  equations in~(2.1)-(2.5)  hold in $\mathbf{G}(j)$ 
when $ 1\leqslant j\leqslant i$.
The following discussions are on the condition that
 the equations in~(2.1)-(2.5) all hold in $\mathbf{G}(j)$ 
when $ 1\leqslant j\leqslant i$.

In the $(i+1)$-th round, the number of new edges that touch the vertex $u$, that is to be
$|\textrm{D}_u(i+1)|-|\textrm{D}_u(i)|$,  is a random variable  with  
$|\textrm{D}_u(i+1)|-|\textrm{D}_u(i)|\sim\textbf{Bin}[|\textrm{F}_u(i)|,p]$. 
Hence, as the equation shown in~(2.2), choosing $t=(np)^{ \frac 23}$ in  Lemma~2.1,
 we have
\begin{align}
&\mathbb{P}\Bigl[|\textrm{D}_u(i+1)|-|\textrm{D}_u(i)|>|\textrm{F}_u(i)| p
+ (np)^{ \frac 23}\Bigr]\notag\\
&<\exp\Biggl[- \frac{(np)^{ \frac 43}}{2\epsilon np\exp[{-i}]+2(np)^{ \frac {11}{6}}p
+ \frac 23(np)^{ \frac 23}}\Biggr]< \frac{1}{n^2},
\end{align}
where the last inequality is clearly true because
$(np)^{ \frac {11}{6}}p/(np)^{\frac 43}=n^{ \frac 12}p^{ \frac 32}=
o(n^{- \frac 14}\log^{\frac 32\alpha(3-r)}n)$ when
$p=o(n^{- \frac 12}\log^{\alpha(3-r)}n)$.

Similarly, the expectation of $|\textrm{X}_{uv}(i+1)|-|\textrm{X}_{uv}(i)|$
depends on whether the open $(r-2,2)$-walks in $\textrm{Y}_{{uv}}(i)$
and $\textrm{Y}_{{vu}}(i)$ are sampled with probability $p$, and the two
open $(r-2,2)$-walks in $\textrm{Z}_{uv}(i)$ are both sampled with probability $p^2$.
As the equations shown in~(2.3) and~(2.5),
$\mathbb{E}[|\textrm{X}_{uv}(i+1)|-|\textrm{X}_{uv}(i)|]=(|\textrm{Y}_{{uv}}(i)|+
|\textrm{Y}_{{vu}}(i)|)p+|\textrm{Z}_{uv}(i)|p^2<3\epsilon np^2\log^{2\alpha(r-2)} n$
and $|\textrm{X}_{uv}(i+1)|-|\textrm{X}_{uv}(i)|$ is stochastically
dominated by a random variable $X$ with $X\sim\textbf{Bin}[3\epsilon n\log^{2\alpha(r-2)} n,p^2]$.
Choosing $t=2\log^{ \frac {2\alpha+1}{2} }n$ in Lemma~2.1,  we have
\begin{align}
&\mathbb{P}\Bigl[|\textrm{X}_{uv}(i+1)|-|\textrm{X}_{uv}(i)|>3\epsilon np^2\log^{2\alpha(r-2)} n %
+ 2\log^{\frac {2\alpha+1}{2}} n\Bigr]\notag\\
&<\exp\Biggl[- \frac{4\log^{2\alpha+1} n}{6\epsilon np^2\log^{2\alpha(r-2)} n
+ \frac 43\log^{ \frac {2\alpha+1}{2}} n}\Biggr]
< \frac{1}{n^3},
\end{align}
where the last inequality is correct because
$np^2\log^{2\alpha(r-2)} n=o(\log^{2\alpha} n)$ 
when $p=o(n^{- \frac 12}\log^{\alpha(3-r)} n)$.

Taking a union of the equation in (2.6) for all vertices and the equation in
(2.7) for any pair of vertices, with probability at least $1-\frac 2n$,   
we have 
\begin{align}
|\textrm{D}_u(i+1)|-|\textrm{D}_u(i)|&\leqslant |\textrm{F}_u(i)|p+(np)^{ \frac {2}{3}}\leqslant
\epsilon np\exp[{-i}]+(np)^{ \frac {11}{6}}p+(np)^{ \frac 23}
\end{align}
 for any  vertex $u\in [n]\backslash\{v_1,\cdots,v_{r-2}\}$ and
 \begin{align}
|\textrm{X}_{uv}(i+1)|-|\textrm{X}_{uv}(i)|&\leqslant
3\epsilon np^2\log^{2\alpha(r-2)} n + 2\log^{ \frac {2\alpha+1}{2}} n=o(\log^{2\alpha} n)
\end{align}
 for any two distinct vertices $u,v\in [n]\backslash\{v_1,\cdots,v_{r-2}\}$,
where the  inequalities in  (2.8) 
and (2.9) are true because of~(2.2) and 
$p=o(n^{- \frac 12}\log^{\alpha(3-r)}n)$ with a constant $\alpha> \frac 12$, respectively.

As  the equations shown in~(2.1) and~(2.8),
it follows that
\begin{align}
|\textrm{D}_u(i+1)|&\leqslant |\textrm{D}_u(i)| +
\epsilon np\exp[{-i}]+(np)^{ \frac {11}{6}}p+(np)^{ \frac 23}\notag\\
&\leqslant \Bigl[1+\bigl(1+ \exp[{-1}]+\cdots +\exp[{-i}]\bigr)\epsilon
+(i+1)n^{ \frac {5}{6}}p^{ \frac{11}{6}}+(i+1)(np)^{ -\frac 13}\Bigr]np\notag\\
&\sim np,
\end{align}
where the last approximate equality is correct because
$1+ \exp[{-1}]+\cdots +\exp[{-i}]= O(1)$, $(i+1)n^{ \frac {5}{6}}p^{ \frac{11}{6}}=o(1)$
and $(i+1)(np)^{ -\frac 13}=o(1)$ when $1\leqslant i<\log n$,
$p\geqslant n^{- \frac 78}$ and $p=o(n^{- \frac 12}\log^{\alpha(3-r)} n)$.
Note that $|\textrm{D}_u(i+1)|\geqslant |\textrm{D}_u(i)|$,
then we have $|\textrm{D}_u(i+1)|\sim np$ for any $u\in [n]\backslash\{v_1,\cdots,v_{r-2}\}$.
As the equations shown in~(2.4) and~(2.9),  we further have
\begin{align}
|\textrm{X}_{uv}(i+1)|&\leqslant |\textrm{X}_{uv}(i)| +
o(\log^{2\alpha} n)\lesssim \log^{2\alpha} n
\end{align}
for any two distinct vertices $u,v\in [n]\backslash\{v_1,\cdots,v_{r-2}\}$.

Note that the open $(r-2,2)$-walks in $\textrm{F}_u(i+1)$
and $\textrm{Y}_{uv}(i+1)$, and two open $(r-2,2)$-walks in $\textrm{Z}_{uv}(i+1)$
all contain at least one new edge in the $(i+1)$-th round. In the following,
assume that the equations in (2.8)-(2.11) hold for any two distinct vertices 
$u,v\in [n]\backslash\{v_1,\cdots,v_{r-2}\}$, we will bound $\textrm{F}_u(i+1)$, $\textrm{Y}_{uv}(i+1)$
and $\textrm{Z}_{uv}(i+1)$ for any two distinct vertices 
$u,v\in [n]\backslash\{v_1,\cdots,v_{r-2}\}$ below.

There are two different ways that a new open $(r-2,2)$-walk with the form $uWv$
for some vertex $v\in[n]$ can appear in $\textrm{F}_u(i+1)$. One approach to generate a new open
$(r-2,2)$-walk $uWv$ in $\textrm{F}_u(i+1)$ is that  the vertex $u$ is connected to one of its neighbors, 
denoted as $w$, by one new edge in the $(i+1)$-th round with 
 at most $|\textrm{D}_u(i+1)|-|\textrm{D}_u(i)|$ ways;  
and the vertex $w$ is connected to one of its neighbors, denoted as $v$, with  
at most $|\textrm{D}_w(i+1)|$ ways; then we take $(r-3)$ common neighbors 
of the vertex $u$ and the vertex $v$ with at most $|\textrm{X}_{uv}(i+1)|^{r-3}$ ways. 
By the equations in (2.8), (2.10) and (2.11), 
 we obtain a new open $(r-2,2)$-walk $uWv$ with at most
$[\epsilon  np \exp[{-i}]+(np)^{ \frac{11}{6}}p+(np)^{ \frac 23}] 
\cdot np\cdot\log^{2\alpha(r-3)} n$ ways.  Another approach to generate a new open
$(r-2,2)$-walk $uWv$ in $\textrm{F}_u(i+1)$ is that the vertex $u$ is connected to
one of its neighbor, denoted as $w$, with at most $|\textrm{D}_u(i+1)|$ ways; 
and the vertex $w$ is connected to one of its neighbors, denoted as $v$,
by one new edge in the $(i+1)$-th round with at most 
$|\textrm{D}_w(i+1)|-|\textrm{D}_w(i)|$ ways; then we further take 
$(r-3)$ common neighbors of the vertex $u$ and the vertex $v$ 
with at most $|\textrm{X}_{uv}(i+1)|^{r-3}$ ways. 
In this situation, we also obtain a new open $(r-2,2)$-walk $uWv$ in $\textrm{F}_u(i+1)$
with at most  $[\epsilon  np \exp[{-i}]+(np)^{ \frac{11}{6}}p+(np)^{ \frac 23}]
\cdot np\cdot\log^{2\alpha(r-3)} n$ ways. Combining these two approaches, it follows that
\begin{align*}
|\textrm{F}_u(i+1)|&\leqslant 2\epsilon  (np)^{2}\exp[{-i}]\log^{2\alpha(r-3)} n+ 
2(np)^{ \frac{17}{6}}p\log^{2\alpha(r-3)} n+2(np)^{ \frac 53}\log^{2\alpha(r-3)} n\\
&<\epsilon n  \exp[{-i-1}]+(np)^{ \frac{11}{6}}
\end{align*}
for any vertex $u\in [n]\backslash\{v_1,\cdots,v_{r-2}\}$,
where the last inequality is correct because $np^{2}\log^{2\alpha(r-3)} n=o(1)$,
$(np)^{ \frac{17}{6}}p\log^{2\alpha(r-3)} n=o((np)^{ \frac{11}{6}})$
and $(np)^{ \frac 53}\log^{2\alpha(r-3)} n=o((np)^{ \frac{11}{6}})$ when
$p\geqslant n^{- \frac 78}$ and $p=o(n^{- \frac 12}\log^{\alpha(3-r)}n)$.

Likewise, in order to obtain a new open $(r-2,3)$-walk with the form $uWw'v$
 in $\textrm{Y}_{uv}(i+1)$, which is composed of a new open $(r-2,2)$-walk $uWw'$
 and the edge $vw'$ in $\textbf{G}(i+1)$,  the vertex $v$
is  connected to one of its neighbors, denoted as $w'$, with at most
 $|\textrm{D}_v(i+1)|$ ways; and  we take one common neighbor of 
 the vertex $u$ and the vertex $w'$ that is generated in the $(i+1)$-th round 
 with at most $|\textrm{X}_{uw'}(i+1)|-|\textrm{X}_{uw'}(i)|$ ways;
 then we further take $(r-3)$ common neighbors of the vertex $u$ and the vertex $w'$  with at most
$|\textrm{X}_{uw'}(i+1)|^{r-3}$ ways. By the equations in (2.9), (2.10) 
and (2.11), we correspondingly obtain a new open $(r-2,3)$-walk $uWw'v$ 
with  $o(np\cdot\log^{2\alpha(r-2)} n)$ ways. 
Hence, we have
\begin{align*}
|\textrm{Y}_{{uv}}(i+1)|&=  o(np\cdot\log^{2\alpha(r-2)} n)<  
\epsilon  np \log^{2\alpha(r-2)} n
\end{align*}
for any two distinct vertices $u,v\in [n]\backslash\{v_1,\cdots,v_{r-2}\}$.

At last, note that a new open $(r-2,4)$-walk with the form $uW_1w'W_2v$ 
 in $\textrm{Z}_{uv}(i+1)$ is obtained by sticking together two new open
 $(r-2,2)$-walks $uW_1w'$ and $w'W_2v$ in the $(i+1)$-th round at the vertex $w'$,
then we enumerate the ways to generate a  new open $(r-2,4)$-walk with the form $uW_1w'W_2v$ below. 
Firstly, in order to obtain a new open $(r-2,2)$-walks $uW_1w'$, 
it also depends on where the new edges in the $(i+1)$-th round 
appear on  $uW_1w'$. By similar analysis with the above case 
 of $|\textrm{F}_u(i+1)|$, as the equations shown in (2.8), (2.10) and (2.11),
 we have a new open $(r-2,2)$-walk $uW_1w'$ with at most $2[\epsilon  np \exp[{-i}]
+(np)^{ \frac{11}{6}}p+(np)^{ \frac 23}]\cdot np\cdot \log^{2\alpha(r-3)}n$ ways.
Next step, in order to obtain a new open $(r-2,2)$-walk $w'W_2v$ in the $(i+1)$-th
 round, we take one common neighbor of the vertex $w'$ and the vertex $v$ that is generated 
 in the $(i+1)$-th round with at most $|\textrm{X}_{vw'}(i+1)|-|\textrm{X}_{vw'}(i)|$ ways;
 then we take other $(r-3)$ common neighbors of the vertex $w'$ and
 the vertex $v$ with  at most $|\textrm{X}_{vw'}(i+1)|^{r-3}$ ways. 
 By the equations in (2.9) and (2.11), we  further have a new open 
 $(r-2,2)$-walks $w'W_2v$ with at most $o(\log^{2\alpha(r-2)}n)$ ways
in the $(i+1)$-th round. Hence, for any two distinct vertices
$u,v\in [n]\backslash\{v_1,\cdots,v_{r-2}\}$, it follows that,
\begin{align*}
|\textrm{Z}_{uv}(i+1)|&\leqslant 2 \Bigl[\epsilon  np \exp[{-i}]
+(np)^{ \frac{11}{6}}p+(np)^{ \frac 23}\Bigr]
\cdot np\cdot o(\log^{2\alpha(2r-5)} n)\\
&=o(n^2p^2\log^{2\alpha(2r-5)} n)\\
&<\epsilon n \log^{2\alpha(r-2)} n, 
\end{align*}
where the last inequality is correct when $p=o(n^{- \frac 12}\log^{\alpha(3-r)} n)$.

Since the equations in (2.8)-(2.11) all hold for any two distinct vertices
$u,v\in [n]\backslash\{v_1,\cdots,v_{r-2}\}$ with conditional probability at least $1- \frac 2n$,
 we now have the equations in (2.1)-(2.5) hold in the $(i+1)$-th round with 
 conditional probability at least $1- \frac 2n$. Let  $\mathcal{A}_i$ be 
 the event that the equations in~(2.1)-(2.5) 
are true in $\mathbf{G}(j)$ for $ 1\leqslant j\leqslant i$, which holds 
with probability at least $(1- \frac 2n)^i$ by induction,
 and $\overline{\mathcal{A}_i}$ be the complement of the event $\mathcal{A}_i$.
By law of total probability 
$\mathbb{P}[\mathcal{A}_{i+1}]=\mathbb{P}[\mathcal{A}_{i+1}|\mathcal{A}_i]\mathbb{P}[\mathcal{A}_i]+
\mathbb{P}[\mathcal{A}_{i+1}|\overline{\mathcal{A}_i}]\mathbb{P}[\overline{\mathcal{A}_i}]$,  
we have $\mathbb{P}[\mathcal{A}_{i+1}]\geqslant (1- \frac{2}{n})^{i+1}$
to complete the proof of induction.
\end{proof}

 We will finish the proof of Theorem 1.1 in the case of $p\geqslant n^{- \frac 78}$ and 
$p=o( n^{- \frac 12}\log^{\alpha(3-r)} n)$ below. Assume  that all equations in Claim B hold.
For an integer $1\leqslant i< \ell$, let $\textrm{F}(i)$ denote the set of all open $(r-2,2)$-walks
in $\textbf{G}(i)-\{v_1,\cdots,v_{r-2}\}$ and let $\textrm{D}(i+1)$ denote 
the set of new edges in the $(i+1)$-th round, 
then $\textrm{D}(i+1)\sim\textbf {Bin}[|\textrm{F}(i)|,p]$ and $\mathbb{E}[|\textrm{D}(i+1)|]= 
\mathbb{E}[\mathbb{E}[|\textrm{D}(i+1)|||\textrm{F(i)}|]]
=p\mathbb{E}[|\textrm{F}(i)|]$ by the tower property.
At the same time, for any new edge, denoted as $uv$, by Claim B,
the number of $(1,2)$-walks containing the edge $uv$, such as $uvw_1$ and $w_2uv$ 
for some $w_1,w_2\in [n]\backslash\{v_1,\cdots,v_{r-2}\}$ in $\mathbf{G}(i+1)$, 
is at most $|\textrm{D}_u(i+1)|+|\textrm{D}_v(i+1)|\sim 2np$,
and the codegrees of the vertex $u$ and the vertex $w_1$, 
the vertex $v$ and the vertex $w_2$ in the generated graph are
$|\textrm{X}_{uw_1}(i+1)|\lesssim\log^{2\alpha} n$ and 
$|\textrm{X}_{vw_2}(i+1)|\lesssim\log^{2\alpha} n$, 
respectively. Hence,  we have
\begin{align*}
\mathbb{E}[|\textrm{F}(i+1)|\bigl||\textrm{D(i+1)}|]
&\lesssim(2 np\log^{2\alpha(r-3)}n)|\textrm{D}(i+1)|,
\end{align*}
and then
\begin{align*}
\mathbb{E}[|\textrm{F}(i+1)|]&=
\mathbb{E}\bigl[\mathbb{E}[|\textrm{F}(i+1)|\bigl||\textrm{D(i+1)}|]\bigr]\\
&\lesssim(2 np\log^{2\alpha(r-3)}n)\mathbb{E}\bigl[|\textrm{D}(i+1)|\bigr]\\
&=(2 np^2\log^{2\alpha(r-3)}n)\mathbb{E}\bigl[|\textrm{F}(i)|\bigr].
\end{align*}
We recursively have $\mathbb{E}[|\textrm{F}(i+1)|]\lesssim
(2 np^2\log^{2\alpha(r-3)}n)^i\mathbb{E}[|\textrm{F}(1)|]$
for $1\leqslant i< \log n$. Note that $\mathbb{E}[|\textrm{F}(1)|]<\epsilon n^2$
by  $|\textrm{F}_u(1)|< \epsilon n$ for any $u\in [n]\setminus \{v_1^*,\cdots,v_{r-2}^*\}$,
then it follows that $\mathbb{E}[|\textrm{F}(\log n)|]<
n^{\log (2np^2\log^{2\alpha(r-3)}n)+2}=o(1)$ when $p=o(n^{-\frac 12}\log^{\alpha(3-r)}n)$.
By Markov's inequality, we indeed obtain that
$\mathbb{P}[|\textrm{F}(\log n)|>0\,|\,\text{Claim B holds}]=o(1)$
 when $p=o(n^{-\frac 12}\log^{\alpha(3-r)}n)$.  
By law of total probability,
we further have $\mathbb{P}[|\textrm{F}(\log n)|>0]=o(1)$, 
which implies that, \textit{w.h.p.}, the process stops after at most $\log n$ rounds.

As the equation $|\textrm{D}_u(i)|\sim np$ shown in Claim B  for any vertex 
 $u\in [n]\backslash \{v_1,\cdots,v_{r-2}\}$ when 
 $1\leqslant i\leqslant \log n$, \textit{w.h.p.}, the final size of 
 added edges in the random $r$-generalized triadic process
is  equal to $ \frac 12n^{2}p(1+o(1))$.

\vskip 0.2cm
\noindent\textbf{Case 2}.\, Consider the case when $p\in (n^{- \frac {33}{24}},n^{ - \frac 78})$. 

\vskip 0.2cm 
\noindent{\bf Claim C}.\, We have, {\it w.h.p.}, $|\textrm{D}_u(1)|< 2n^{ \frac 18}$,
$|\textrm{F}_u(1)|< 2^{r-1}n^{ \frac 14}\log^{r-3}n$,
$|\textrm{Y}_{{uv}}(1)| <2^{r-1}n^{ \frac 18}\log^{r-2}n$,
$|\textrm{X}_{uv}(1)|< 2\log n$ and $|\textrm{Z}_{uv}(1)|<2^{2r-3}n^{ \frac 14}\log^{2r-5}n$
hold for any two vertices $u,v\in [n]\backslash\{v_1,\cdots,v_{r-2}\}$,
respectively.

\begin{proof}[Proof of Claim C]\, Since $\mathbb{E}[|\textrm{D}_u(1)|]\sim np< n^{ \frac 18}$
and $\mathbb{E}[|\textrm{X}_{uv}(1)|]\sim np^2<n^{- \frac 34}$,
choosing $t=  n^{ \frac 18}$ and $t= \frac 32\log n$ in  Lemma 2.1,
taking a union bound, we also obtain, {\it w.h.p.},
 $|\textrm{D}_u(1)|< 2n^{ \frac 18}$ and $|\textrm{X}_{uv}(1)|< 2\log n$
for any two distinct vertices $u,v\in [n]\backslash\{v_1,\cdots,v_{r-2}\}$, respectively.
Following the same analysis in Claim A,  {\it w.h.p.}, 
we also have all the equations of $|\textrm{F}_u(1)|$, 
$|\textrm{Y}_{{uv}}(1)|$ and $|\textrm{Z}_{uv}(1)|$  in Claim C hold
for any two distinct vertices $u,v\in [n]\backslash\{v_1,\cdots,v_{r-2}\}$.
\end{proof}

We now show that, by a shorter proof than Case 1, \textit{w.h.p.}, no edges are added 
during round three and the number of added edges during round two is
negligible compared to the added one in $\textbf{G}(1)$.

Assume that all equations in Claim C hold. Thus,
we have $\mathbb{E}\bigl[|\textrm{D}_u(2)|-|\textrm{D}_u(1)|\bigr]
=|\textrm{F}_u(1)|p< 2^{r-1}n^{ - \frac 58}\log^{r-3}n=o(1)$
for any vertex $u\in [n]/\{v_1,\cdots,v_{r-2}\}$.
By Lemma~2.1, choosing $t=n^{ \frac{1}{16}}$,
$\mathbb{P}[|\textrm{D}_u(2)|-|\textrm{D}_u(1)|>|\textrm{F}_u(1)|p
+ n^{ \frac{1}{16}}]< \frac{1}{n^2}$.
Taking a union of the above equation for all vertices,  
  {\it w.h.p.}, $|\textrm{D}_u(2)|-|\textrm{D}_u(1)|< 2n^{ \frac{1}{16}}$
  holds for any vertex $u\in [n]/\{v_1,\cdots,v_{r-2}\}$,
which implies that, \textit{w.h.p.}, no vertex receives more than $2n^{  \frac{1}{16}}$ new edges 
during round two and $|\textrm{D}_u(2)|<3n^{ \frac 18}$ for any 
 $u\in [n]\backslash\{v_1,\cdots,v_{r-2}\}$.
We also have $\mathbb{P}[\sum_{u}
(|\textrm{D}_u(2)|-|\textrm{D}_u(1)|)>n^{  \frac{7}{16}}]\leqslant
\sum_{u}\mathbb{E}[|\textrm{D}_u(2)|-|\textrm{D}_u(1)|]/n^{  \frac{7}{16}}
<2^{r-1}n^{- \frac{1}{16}}\log^{r-3}n=o(1)$
by Markov's inequality, which implies that, {\it w.h.p.},
at most $n^{  \frac{7}{16}}$ vertices receive at least one
edge during round two. Overall, we have that, \textit{w.h.p.}, 
at most $2n^{ \frac{1}{2}}$ new edges are added during round two.

By $|\textrm{Y}_{{uv}}(1)|+|\textrm{Y}_{{vu}}(1)|< 2^{r}n^{ \frac 18}\log^{r-2} n$
and $|\textrm{Z}_{uv}(1)|<2^{2r-3}n^{ \frac 14}\log^{2r-5} n$ as the equations shown in Claim C,
it follows that $\mathbb{E}[|\textrm{X}_{uv}(2)|-|\textrm{X}_{uv}(1)|]<
2^{r+1}n^{  \frac 18}p\log^{r-2} n$ because 
$2^{2r-3}n^{ \frac 14}p^2\log^{2r-5} n\ll$ $2^{r}n^{  \frac 18}p\log^{r-2} n$
when $p<n^{- \frac 78}$, and 
$|\textrm{X}_{uv}(2)|-|\textrm{X}_{uv}(1)|$ is stochastically dominated by
a random variable $X$ with
$X\sim\textbf{Bin}[2^{r+1}n^{  \frac 18}\log^{r-2} n,p]$. Choosing $t=2\log n$ in Lemma~2.1 and
taking a union bound, by $|\textrm{X}_{uv}(1)|<2\log n$ in Claim C, we also obtain that, 
\textit{w.h.p.}, $|\textrm{X}_{uv}(2)|<4\log n$ 
for any two distinct
vertices $u,v\in [n]\backslash\{v_1,\cdots,v_{r-2}\}$.

By law of total probability again, 
we finally have, \textit{w.h.p.}, there are at most $2n^{ \frac 12}\cdot 
(3n^{ \frac 18})\cdot(4\log n)^{r-3}\leqslant {6}^{r-2}n^{ \frac 58}\log^{r-3} n$
open $(r-2,2)$-walks in $\textbf{G}(2)$. Once again,
 the expected number of edges added during round three is at most
$6^{r-2}n^{ - \frac 14}\log^{r-3}n=o(1)$ when $p<n^{- \frac 78}$, and \textit{w.h.p.}, 
we have no edges are added during round three.
Note that
$2n^{ \frac 12}\ll \frac{1}{2}n^2 p$ when $p> n^{- \frac{33}{24}}$, 
which implies that the number of added edges during round two is negligible compared
to the added one in $\textbf{G}(1)$.

\vskip 0.2cm
\noindent\textbf{Case 3}.\, Consider the case when $p\leqslant n^{- \frac {33}{24}}$.
\vskip 0.2cm
By Lemma~2.1, \textit{w.h.p.}, the number of added edges in $\textbf{G}(1)$ 
is at most $2n^{ \frac {5}{8}}$ and $|\textrm{X}_{uv}(1)|< 4\log n$ 
for any two distinct vertices $u,v\in [n]/\{v_1,\cdots,v_{r-2}\}$.
Hence, \textit{w.h.p.}, the number of open $(r-2,2)$-walks in $\textbf{G}(1)$ is at most
$4^{r-2}n^{ \frac{5}{4}}\log^{r-3} n$. Under this condition, the
expected number of added edges during round two is 
at most $4^{r-2}n^{- \frac 18}\log^{r-3} n$ when $p\leqslant n^{- \frac {33}{24}}$. By law of total probability
and Markov's inequality again,
\textit{w.h.p.}, no edges are added 
during round two.

\vskip 0.2cm
Combining these three cases, after removing $\{v_1,\cdots,v_{r-2}\}$, \textit{w.h.p.},
the generated graph in the random $r$-generalized triadic process
essentially behaves like $\mathbb{G}(n,p)$ when $p=o(n^{- \frac 12}\log^{\alpha(3-r)} n)$ 
with a constant $\alpha> \frac 12$, and
the final number of added edges in the process approximately equals  $ \frac 12n^2p$
when $p=\omega(n^{-2})$ by Lemma~2.1 because $n^2p\rightarrow\infty$.
\end{proof}

It is easy to obtain that
$p= n^{-1}\log n$ is also the threshold of connectivity
for the random $r$-generalized triadic process.

\begin{proof}[Proof of Theorem~1.3]
Assuming removal of $\{v_1,\cdots,v_{r-2}\}$,
during the further evolution of the process, let
$G'$ be any graph generated after round one. Then, the graph $G'$
is a subgraph of the transitive closure of $\textbf{G}(1)-\{v_1,\cdots,v_{r-2}\}$,
and $G'$ is connected if and only if $\textbf{G}(1)-\{v_1,\cdots,v_{r-2}\}$
is connected. Since $\textbf{G}(1)-\{v_1,\cdots,v_{r-2}\}$
behaves like $\mathbb{G}(n,p)$, it implies that the threshold for connectivity of
the $r$-generalized triadic process is $ n^{-1}\log n$.
\end{proof}



\end{document}